\documentclass[11pt,reqno]{amsart}

\usepackage[left=3.5cm,right=3.5cm,top=3.5cm,bottom=3.5cm]{geometry}
\usepackage{tipa}
\usepackage{amssymb}
\usepackage{graphicx}
\usepackage[hidelinks]{hyperref}
\usepackage{enumerate}
\usepackage{xcolor}
\usepackage{amsmath}
\allowdisplaybreaks[4]

\numberwithin{equation}{section}
\newtheorem{theorem}{Theorem}[section]

\newtheorem{corollary}{Corollary}[section]
\newtheorem{lemma}[theorem]{Lemma}

\theoremstyle{definition}

\newtheorem*{remarks*}{Remarks}

\numberwithin{equation}{section}

\title{A Generalization of A Result of Gauss on Primitive Root}

\author[H. Zhong]{Hao Zhong}
\address{(H. Zhong) The School of Information Science and Technology, Chengdu University of Technology, Chengdu, 610059, China}
\email{zhonghao@cdut.edu.cn}
\thanks{}

\author[T. Cai]{Tianxin Cai}
\address{(T. Cai) Department of Mathematics, Zhejiang University, Hangzhou, 310027, China}
\email{txcai@zju.edu.cn}
\thanks{}

\date{2019/10/24}


\begin{document}

\maketitle

\thispagestyle{empty}

\begin{abstract}
A primitive root modulo an integer $n$ is the generator of the multiplicative group of integers modulo $n$. Gauss proved that for any prime number $p$ greater than $3$, the sum of its primitive roots is congruent to $1$ modulo $p$ while its product is congruent to $\mu(p-1)$ modulo $p$, where $\mu$ is the M\"{o}bius function. In this paper, we will generalize these two interesting congruences and give the congruences of the sum and the product of integers with the same index modulo $n$.
\end{abstract}

\section{Introduction}

Let $a$ and $n$ be two relatively prime integers. The index of $a$ modulo $n$, denoted by $ind_{n}(a)$, is the smallest positive number $k$ such that $a^{k}\equiv 1\pmod{n}$. If $ind_{n}(a)=\phi(n)$, where $\phi$ is the Euler's totient function, then we say $a$ is a primitive root modulo $n$. Primitive roots are widely used in cryptography, such as attacking the discrete log problem, key exchange problem and many other public key cryptosystem. The primitive root theorem identifies all moduli which primitive roots exist, that is,
$$1,\ 2,\ 4,\ p^{\alpha}\ \text{and }2p^{\alpha}$$
where $p$ is an odd prime and $\alpha$ is a positive integer. For an odd prime $p>3$, Gauss obtained the following congruences in Article $81$,

\begin{theorem}[Gauss]\label{gausspr1}
$$\prod_{g\text{ is a primitive root mod }p}g\equiv1\pmod p.$$
\end{theorem}

\begin{theorem}[Gauss]\label{gausspr2}
$$\sum_{g\text{ is a primitive root mod }p}g\equiv\mu(p-1)\pmod p.$$
\end{theorem}

Few papers have been published concerning these two awesome congruences until Cai \cite{cai2017book} recently proved

\begin{theorem}[Cai]\label{caipr1}
Let $p$ be a prime greater than $3$ and $\alpha$ a positive integer. Then
$$\prod_{g\text{ is a primitive root mod }p^{\alpha}}g\equiv1\pmod{p^{\alpha}}.$$
\end{theorem}

\begin{theorem}[Cai]\label{caipr2}
Let $p$ be an odd prime and $\alpha$ a positive integer. Then
$$\sum_{g\text{ is a primitive root mod }p^{\alpha}}g\equiv\mu(p-1)\phi(p^{\alpha-1})\pmod{p^{\alpha}}.$$
\end{theorem}

We will continue Cai's work and give many interesting congruences concerning the sum and product of integers with same index. In this article, $\alpha$ and $\beta$ represent positive integers; $m$ is some given positive intger; $U_m$ is the set of invertible elements in $\mathbb{Z}/m\mathbb{Z}$; $\lambda(m)$ is the Carmichael function, namely, $\lambda(m)=$ min$\{k>0:$ $a^k\equiv1\pmod m$ for any $a\in U_m$ $\};$ $\delta$ is a positive divisor of $\lambda(m)$.

It is not a tough work to obtain the product version as follows.

\begin{theorem}\label{th:2.3.1}
\begin{equation}\label{eq:2.3.1}
\prod_{\substack{a\in U_m\\ind_m(a)=\delta}}a\equiv\begin{cases}-1\pmod m,&\text{ If }\delta=2\text{ and }m\text{ has a primitive root,}\\1\pmod m,&\text{ otherwise.}\end{cases}
\end{equation}
\end{theorem}

As for the sum, it is not a simple congruence written in one line. However, for fixed $\delta$, one can easily get the following results.

\begin{theorem}\label{th:2.3.2}
(1) $$\sum_{\substack{a\in U_m\\ind_m(a)=1}}a\equiv 1\pmod m;$$
(2) $$\sum_{\substack{a\in U_m\\ind_m(a)=2}}a\equiv -1\pmod m;$$
(3) $$\sum_{\substack{a\in U_m\\ind_m(a)=\delta\\4\mid \delta}}a\equiv 0\pmod m.$$
\end{theorem}

By applying this theorem, we obtain a result related to Fermat primes.

\begin{corollary}
If the regular $m$-gon can be constructed by compass and straightedge, then

$$\sum_{\substack{a\in U_{m}\\ind_{m}(a)=\delta}}a\equiv 0\pmod m.$$

\end{corollary}

To give the congruence mod general integer $m$, we first list the ones with some special moduli. For integers with primitive roots, we have

\begin{theorem}\label{th:2.3.3}
(1) $$\sum_{\substack{a\in U_m\\ind_2(a)=\delta}}a\equiv 1\pmod 2;$$
(2) $$\sum_{\substack{a\in U_m\\ind_4(a)=\delta}}a\equiv (-1)^{\delta+1}\pmod 4;$$
(3) $$\sum_{\substack{a\in U_m\\ind_8(a)=\delta}}a\equiv (-1)^{\delta+1}\pmod 8;$$
(4) If $\alpha>3$, $\delta=2^{\beta}$ and $1<\beta\leq\alpha-2$. Then
$$\sum_{\substack{a\in U_{2^{\alpha}}\\ind_{2^{\alpha}}(a)=\delta}}a\equiv 0\pmod{2^{\alpha}};$$
(5) Let $p$ be an odd prime. Then
$$\sum_{\substack{a\in U_{p^{\alpha}}\\ind_{p^{\alpha}}(a)=\delta}}a\equiv \mu((\delta,p-1))\phi(p^{ord_{p}(\delta)})\pmod{p^{\alpha}}.$$
where $ord_{p}(n)=max\{k:p^k\mid n\}.$
\end{theorem}

For a product of two distinct odd primes, we have

\begin{theorem}\label{th:2.3.4}
Let $p$ and $q$ be two distinct odd primes. Then
$$\sum_{\substack{a\in U_{pq}\\ind_{pq}(a)=\delta}}a\equiv \mu((\delta,p-1))I((\frac{\delta}{(\delta,p-1)},p-1))\frac{\phi(\delta)}{\phi((\delta,p-1))}\pmod p.$$
\end{theorem}

We end this section with the following congruence.

\begin{theorem}\label{th:2.3.5}

If $m=\prod_{i=1}^{k}p^{\alpha_i}$, then for any $p^{\alpha}\parallel m$,
\begin{align}\label{eq:th:2.3.5.2}
\sum_{\substack{a\in U_{m}\\ind_{m}(a)=\delta}}a
\equiv\mu((\delta,p-1))I((\frac{\delta}{(\delta,p-1)},p-1))F(\phi(p_1^{\alpha_1}),\cdots,\phi(p_k^{\alpha_k})&;\frac{\delta}{(\delta,p-1)})\\\nonumber&\pmod{p^{\alpha}}.
\end{align}
where $F(a_1,\cdots,a_k;n)=\prod_{p^{\alpha}\parallel n}\prod_{ord_p(a_i)<\alpha}p^{ord_{p}(a_i)}\prod_{ord_p(a_i)\geq\alpha}\phi(p^{\alpha})$.
\end{theorem}

Though this result seems complex, it shows that there is no need to compute the sum of integers with same index modulo $m$ by finding each of them.

\section{Preliminaries}

The following two lemmas are well known facts in the study of primitive roots. (See \cite{apostol}.)

\begin{lemma}\label{apostolind1}
Let $a$ belong to $U_{m}$. Then for any positive integer $k$,
$$ind_m(a^k)=\frac{ind_{m}(a)}{(k,ind_{m}(a))}.$$
\end{lemma}

\begin{lemma}\label{apostolind2}
Let $m$, $n$ and $a$ be positive integers, relatively prime in pairs. Then $ind_{mn}(a)=[ind_{m}(a),ind_{n}(a)]$. In particular, $\lambda(mn)=[\lambda(m),\lambda(n)],$ where $[m,n]$ represents the least common multiple of $n$ and $m$.
\end{lemma}

The Dirichlet convolution defined on the set of arithmetical functions: $\mathcal{A}:=\{f:\mathbb{Z}^{+}\to\mathbb{C}\}$ is very important in number theory, especially in analytic number theory. Two arithmetical functions $f$ and $g$ are combined by the convolution as follows
$$(f*g)(n)=\sum_{ab=n}f(a)g(b).$$
All the functions in $\mathcal{A}$ that are invertible form a group $\mathcal{A}^{*}$, which contains an essential subgroup called the set of the multiplicative functions: $\mathcal{M}:=\{f\in\mathcal{A}:f\not\equiv0, f(mn)=f(m)f(n)\text{ for any two relatively prime positive integers }m\text{ and }n\}.$ An interesting result about Dirichlet convolution is the M\"{o}bius inversion formula showing that $f=g*u$ if and only if $f*\mu=g$ where $u\equiv 1$ and $\mu$ is the M\"{o}bius function.

Besides Dirichlet convolution, many other arithmetical convolutions also have long been studied by mathematicians. (See \cite{Subbarao, Lehmer1}.) To prove the results in this paper, we introduce one of them. The $lcm$ convolution of $f$ and $g$ is defined as follows
$$f\circ g=\sum_{[a,b]=n}f(a)g(b).$$
It is obvious that $\mathcal{M}$ is closed under this binary operation. Lehmer \cite{Lehmer1} proved that
\begin{lemma}\label{lem:lehmerconvolution2}
Let $$M(1)=1,\qquad M(n):=\prod_{p^{\alpha}\parallel n}\big((\alpha+1)^{-1}-{\alpha}^{-1}\big).$$
Then $M\circ u=I$ where $I(n)=\lfloor1/n\rfloor$. Thus $f=g\circ u$ if and only if $f\circ M=g.$
\end{lemma}

In another paper, Lehmer \cite{Lehmer2} gave the following formula
\begin{lemma}\label{lem:lehmerconvolution4}
$$(f*u)(g*u)=(f\circ g)*u$$
for any arithmetical functions $f$ and $g$.
\end{lemma}

By the definition of $lcm$ convolution, Lemma \ref{lem:lehmerconvolution2} and Lemma \ref{lem:lehmerconvolution4}, one can easily obtain the following lemmas.

\begin{lemma}
(1) $f\circ\mu=f(1)\mu$ for any $f$ in $\mathcal{A}$.

(2) $f \circ g=I$ if and only if $\Big(\sum_{d\mid n}f(d)\Big)\Big(\sum_{d\mid n}g(d)\Big)=1$ for any positive integer $n$.

(3) $$(f\circ g)(p^{\alpha})=f(p^{\alpha})\sum_{i=0}^{\alpha-1}g(p^i)+g(p^{\alpha})\sum_{j=0}^{\alpha-1}f(p^j)+f(p^{\alpha})g(p^{\alpha})$$
where $p$ is a prime number.

(4) Let $f$ be a multiplicative function. Then $$\sum_{[a,b]=n}\mu(a)f(a)=\prod_{p^2\mid n}(1-f(p))\prod_{p\parallel n}(1-2f(p)).$$ In particular, if $n$ is a perfect square, then $\mu f*u=\mu f\circ u.$
\end{lemma}

\section{proofs of the theorems}

\begin{proof}[Proof of Theorem \ref{th:2.3.1}]
If $\delta=1$, then $a=1$. Thus \eqref{eq:2.3.1} is obvious.

If $\delta=2$, then $m>2$. Assume that $m=2^{\alpha}\prod_{i=1}^{k}p_{i}^{\beta_i}$ where $\{p_i\}$ is a set of prime numbers that is pairwise coprime. By the Chinese remainder theorem,
\begin{eqnarray*}
&\#&\{x\in U_m:x^2\equiv1\pmod m\}\\
=&\#&\{x\in U_{2^{\alpha}}:x^2\equiv1\pmod{2^{\alpha}}\}\prod_{i=1}^{k}\#\{x\in U_{p^{\beta_i}}:x^2\equiv1\pmod{p_{i}^{\beta_i}}\}.
\end{eqnarray*}
Hence,
\begin{equation*}
\#:=\#\{x\in U_m:x^2\equiv1\pmod m\}=\begin{cases}2^k,&\text{ if}\alpha=0\text{ or }1\\2^{k+1},&\text{ if}\ alpha=2\\2^{k+2},&\text{ if}\alpha\geq3.\end{cases}
\end{equation*}

Notice that if $a\not\equiv\pm1\pmod m$, then $ind_m(a)=2$ if and only if $ind_m(-a)=2$. Since $a$ is coprime to $m$, $a\not\equiv -a\pmod m$ if $m>2$. Besides, $ind_m(a)=2$ implies $-a^2\equiv-1\pmod m$. Therefore,
\begin{eqnarray*}
\prod_{\substack{a\in U_m\\ind_m(a)=2}}a&\equiv&(-1)^{\#/2}\pmod m\\
&\equiv&\begin{cases}-1\pmod m,&\text{ if}\alpha=0,1,k=1\text{ or }\alpha=2,k=0,\\1\pmod m,&\text{ otherwise.}\end{cases}
\end{eqnarray*}
The product is congruent to $-1$ modulo $m$ implies $m$ has a primitive root. Thus, \eqref{eq:2.3.1} holds for $\delta=2$.

For $\delta>2$, $ind_m(a)=\delta$ implies $ind_m(a^{-1})=\delta$ where $a^{-1}$ is the multiplicative inverse of $a$ in $U_m$. By the definition of $ind$, $a\not\equiv a^{-1}\pmod m$. Therefore, all integers with $ind$ $delta$ can be arranged in pairs with product $1$ modulo $m$. This proves the theorem.
\end{proof}

\begin{proof}[Proof of Theorem \ref{th:2.3.2}]
(1)is obvious.

By the proof of Theorem \ref{th:2.3.1}, if $a\not\equiv\pm1\pmod m$, then $ind_m(a)=2$ if and only if $ind_m(-a)=2$. Therefore, with the exception of $-1$, all integers with $ind$ $2$ can be arranged in pairs with sum $0$ modulo $m$, which proves (2).

If $4$ divides $\delta$, then $(-a)^{\delta}\equiv1\pmod m$ for $a$ such that $ind_m(a)=\delta$. Let $\delta'$ denote $ind_m(-a)$. If $\delta'<\delta$, then $(-1)^{\delta'}\equiv-1\pmod m$. Otherwise, $a^{\delta'}\equiv1\pmod m$, which contradicts the definition of $ind$. Thus, $\delta'$ is odd and $\delta'\mid\delta$. This implies $a^{2\delta'}\equiv-1\pmod m$. Moreover, $\delta\mid2\delta'$. Hence, $\delta=2\delta'$ which contradicts the assumption that $4$ divides $\delta$. Therefore, $ind_m(a)=\delta$ if and only if $ind_m(-a)=\delta$. For $m>4$, all integers with $ind$ $\delta$ can be arranged in pairs with sum $0$ modulo $m$, which proves (3).
\end{proof}

\begin{proof}[Proof of Theorem \ref{th:2.3.3}]
One can check the congruences in (1), (2) and (3) easily. (4) is a straightforward corollary to Theorem \ref{th:2.3.2}(3).

We now prove the result in (5).

Since $m=p^{\alpha}$, $m$ has a primitive root denoted by $g$. By Lemma \ref{apostolind1}, for any positive integer $d$ that divides $\phi(m)$, there exists integer $a$ such that $ind_m(a)=d$. For any integer $a$ such that $ind_m(a)=\delta$, if $(k,\delta)=1$, then $ind_m(a^k)=\delta$. Notice that $a^i\equiv a^j\pmod m$ if and only if $i\equiv j\pmod{\delta}$. Thus, there exists at least $\phi(\delta)$ integers with $ind$ $\delta$ modulo $m$. Since $n=\sum_{d\mid n}\phi(d)$, there exists exactly $\phi(\delta)$ integers with $ind$ $\delta$ modulo $m$.

Hence, for any integer $a$ such that $ind_m(a)=\delta$,

\begingroup
\allowdisplaybreaks
\begin{eqnarray*}
\sum_{\substack{a\in U_{p^{\alpha}}\\ind_{p^{\alpha}}(a)=\delta}}a&\equiv&\sum_{\substack{j=1\\(j,\delta)=1}}^{\delta}a^j\equiv \sum_{j=1}^{\delta}a^j\sum_{\substack{d\mid j\\d\mid \delta}}\mu(d)\\
&\equiv&\sum_{d\mid\delta}\mu(d)\sum_{\substack{j=1\\d\mid j}}^{\delta}a^j\equiv\sum_{d\mid\delta}\mu(d)\sum_{j=1}^{\delta/d}a^{jd}\\
&\equiv&\sum_{d\mid\delta}\mu(d)a^d\frac{a^{\delta}-1}{a^d-1}\equiv (a^{\delta}-1)\sum_{d\mid\delta}\mu(d)(1+\frac{1}{a^d-1})\\
&\equiv&(a^{\delta}-1)\sum_{d\mid\delta}\mu(d)+\sum_{d\mid\delta}\mu(d)\frac{a^{\delta}-1}{a^d-1}\\
&\equiv&\sum_{d\mid\delta}\mu(d)\frac{a^{\delta}-1}{a^d-1}\equiv\sum_{\substack{d\mid\delta\\a^d\equiv1\pmod p}}\mu(d)\frac{a^{\delta}-1}{a^d-1}\\
&\equiv&\sum_{\substack{d\mid\delta\\g^{\phi(m)d/\delta}\equiv1\pmod p}}\mu(d)\frac{g^{\phi(m)}-1}{g^{\phi(m)d/\delta}-1}\pmod m.
\end{eqnarray*}
\endgroup
Notice that the following four are equivalent
$$g^{\phi(m)d/\delta}\equiv1\pmod p,$$ $$\phi(m)d/\delta\equiv0\pmod{p-1},$$ $$(p-1)d/\delta\equiv0\pmod{p-1},$$ $$(\delta/d,p-1)=1.$$
Thus,
\begin{eqnarray*}
\sum_{\substack{a\in U_{p^{\alpha}}\\ind_{p^{\alpha}}(a)=\delta}}a&\equiv&\sum_{\substack{d\mid\delta\\(\delta/d,p-1)=1}}\mu(d)\sum_{j=1}^{\delta/d}\binom{\delta/d}{j}(g^{\phi(m)d/\delta}-1)^{j-1}\\
&\equiv&\sum_{d\mid p^{ord_p(\delta)}}\mu(d(\delta,p-1))\sum_{j=1}^{\delta/(d(\delta,p-1))}\binom{\delta/(d(\delta,p-1))}{j}(g^{\phi(m)d(\delta,p-1)/\delta}-1)^{j-1}\\
&\equiv&\mu((\delta,p-1))\sum_{d\mid p^{ord_p(\delta)}}\mu(d)\\
&&\times\sum_{j=1}^{\delta/(d(\delta,p-1))}\binom{\delta/(d(\delta,p-1))}{j}(g^{\phi(m)d(\delta,p-1)/\delta}-1)^{j-1}\pmod m.
\end{eqnarray*}
(5) holds apparently for $ord_p(\delta)=0$. For $r:=ord_p(\delta)>0$,
\begin{eqnarray*}
\sum_{\substack{a\in U_{p^{\alpha}}\\ind_{p^{\alpha}}(a)=\delta}}a&\equiv&\mu((\delta,p-1))\{\sum_{j=1}^{p^r}\binom{p^r}{j}(g^{\phi(p^{\alpha-r})}-1)^{j-1}-\sum_{j=1}^{p^{r-1}}\binom{p^{r-1}}{j}(g^{\phi(p^{\alpha-r+1})}-1)^{j-1}\}\\
&\equiv&\mu((\delta,p-1))(p^r-p^{r-1})\equiv\mu((\delta,p-1))\phi(p^r)\pmod m
\end{eqnarray*}
which completes the proof.
\end{proof}

\begin{proof}[Proof of Theorem \ref{th:2.3.4}]
Since $p$ and $q$ are two relatively prime prime numbers, there exists two integers $x$ and $y$ such that $px+qy=1$. Thus,
\begin{eqnarray*}
\sum_{\substack{a\in U_{pq}\\ind_{pq}(a)=\delta}}a&\equiv&\sum_{\substack{a=1\\ind_{p}(a)=\delta_1,ind_{q}(a)=\delta_2\\ [\delta_1,\delta_2]=\delta}}^{pq}a \equiv \sum_{\substack{1\leq i\leq p,1\leq j\leq q\\ind_{p}(i)=\delta_1,ind_{q}(j)=\delta_2\\ [\delta_1,\delta_2]=\delta}}(pxj+qyi)\\
&\equiv&\sum_{[\delta_1,\delta_2]=\delta}\sum_{\substack{i=1\\ind_{p}(i)=\delta_1}}^{p}\sum_{\substack{j=1\\ind_{q}(j)=\delta_2}}^{q}i\equiv\sum_{[\delta_1,\delta_2]=\delta}\sum_{\substack{i=1\\ind_{p}(i)=\delta_1}}^{p}i\sum_{\substack{j=1\\ind_{q}(i)=\delta_2}}^{p}1\\
&\equiv&\sum_{\substack{[\delta_1,\delta_2]=\delta\\ \delta_1\mid p-1,\delta_2\mid q-1}}\mu(\delta_1)\phi(\delta_2)\pmod p.
\end{eqnarray*}

Let $$\chi_{n}(a):=\begin{cases}1,&\text{ if }a\mid n,\\0,&\text{ if }a\nmid n,\\\end{cases}.$$
Then
\begin{eqnarray*}
\sum_{\substack{a\in U_{pq}\\ind_{pq}(a)=\delta}}a&\equiv&\sum_{[\delta_1,\delta_2]=\delta}\mu(\delta_1)\chi_{p-1}(\delta_1)\phi(\delta_2)\chi_{q-1}(\delta_2)\\
&\equiv&\mu\chi_{p-1}\circ\phi\chi_{q-1}(\delta)\\
&\equiv&\sum_{d\mid \delta}\mu(\delta/d)\sum_{x\mid d}\mu(x)\chi_{p-1}(x)\sum_{y\mid d}\phi(y)\chi_{p-1}(y)\\
&\equiv&\sum_{\delta_1\mid(\delta,p-1)}\mu(\delta_1)\sum_{\substack{\delta_2\mid q-1\\ [\delta_1,\delta_2]=\delta}}\phi(\delta_2)\pmod p.
\end{eqnarray*}

If $(\frac{\delta}{(\delta,p-1)},p-1)\neq1$, then there exists some prime $r$ that divides $(\frac{\delta}{(\delta,p-1)},p-1)$.
Hence, $ord_r(\delta)>ord_r((\delta,p-1))$. Moreover, $ord_r(\delta)>ord_{r}(p-1)>0.$

Let $\delta_1$ be a divisor of $(\delta,p-1)/r$. Then $ord_{r}(\delta_1)<ord_{r}(p-1)<ord_{r}(\delta)$. Thus, $ord_{r}(r\delta_1)<ord_{r}(\delta).$ Therefore, $[\delta_1,\delta_2]=\delta$ if and only if $[r\delta_1,\delta_2]=\delta$, which leads to the following
$$\sum_{\delta_1\mid(\delta,p-1)}\mu(\delta_1)\sum_{\substack{\delta_2\mid q-1\\ [\delta_1,\delta_2]=\delta}}\phi(\delta_2)=\sum_{\delta_1\mid(\delta,p-1)/r}(\mu(\delta_1)+\mu(r\delta_1))\sum_{\substack{\delta_2\mid q-1\\ [\delta_1,\delta_2]=\delta}}\phi(\delta_2)=0.$$
Therefore we proved the case of $(\frac{\delta}{(\delta,p-1)},p-1)>1$.

Assume $(\frac{\delta}{(\delta,p-1)},p-1)=1$. If $(\delta,p-1)=1$, then
$$\sum_{\delta_1\mid(\delta,p-1)}\mu(\delta_1)\sum_{\substack{\delta_2\mid q-1\\ [\delta_1,\delta_2]=\delta}}\phi(\delta_2)=\phi(\delta).$$
This proves the theorem.

If $(\delta,p-1)>1$, then $ord_{r}(\delta)\leq ord_{r}(p-1)$ for any prime $r$ that divides $(\delta,p-1)$. Thus, $(\frac{\delta}{(\delta,p-1)},\delta_1)=1$. Hence,
\begin{eqnarray*}
\sum_{\delta_1\mid(\delta,p-1)}\mu(\delta_1)\sum_{\substack{\delta_2\mid q-1\\ [\delta_1,\delta_2]=\delta}}\phi(\delta_2)&=&\sum_{\delta_1\mid(\delta,p-1)}\mu(\delta_1)\sum_{\substack{\delta_2\mid q-1\\ [\delta_1,\delta_2]=(\delta,p-1)}}\phi(\delta_2)\phi\Big(\frac{\delta}{(\delta,p-1)}\Big)\\
&=&\frac{\phi(\delta)}{\phi((\delta,p-1))}\mu\circ\phi\chi_{q-1}((\delta,p-1))\\
&=&\frac{\phi(\delta)}{\phi((\delta,p-1))}\mu((\delta,p-1))\phi\chi_{q-1}(1)\\
&=&\frac{\phi(\delta)}{\phi((\delta,p-1))}\mu((\delta,p-1)).
\end{eqnarray*}
This completes the proof.
\end{proof}

\begin{proof}[Proof of Theorem \ref{th:2.3.5}]

By Theorem \ref{th:2.3.3} (5) and the methods mentioned in the proof of Theorem \ref{th:2.3.4}, we can obtain the following congruence
\begin{equation}\label{final}
\sum_{\substack{a\in U_{m}\\ind_{m}(a)=\delta}}a\equiv \big(\mu\chi_{p-1}\circ\phi\chi_{\phi(p_1^{\alpha_1})}\circ\cdots\circ\phi\chi_{\phi(p_k^{\alpha_k})}\big)(\delta)\pmod{p^{\alpha}}.
\end{equation}
Now, we will show how \eqref{final} leads to \eqref{eq:th:2.3.5.2}. Let $rad(n):=\prod_{p\mid n}p$. Then for any positive integer $a$ and multiplicative function $g$,

\begin{eqnarray*}
\mu\chi_{a}\circ g(n)&=&\prod_{p^{\alpha}\parallel n}\{I((p^{\alpha},a))g*u(p^{\alpha})-I((p^{\alpha-1},a))g*u(p^{\alpha-1})\}\\
&=&I((\frac{n}{rad(n)},a))\prod_{\substack{p\parallel n\\p\mid a}}(-1)g*u(1)\prod_{\substack{p^{\alpha}\parallel n\\p\nmid a}}(g*u(p^{\alpha})-g*u(p^{\alpha-1}))\\
&=&I((\frac{n}{rad(n)},a))\prod_{\substack{p\parallel n\\p\mid a}}(-1)\prod_{p^{\alpha}\parallel \frac{n}{(n,a)}}(\mu(1)g*u(p^{\alpha})+\mu(p)g*u(p^{\alpha-1}))\\
&=&\mu((n,a))I((\frac{n}{(n,a)},a))\prod_{p^{\alpha}\parallel \frac{n}{(n,a)}}\mu*g*u(p^{\alpha})\\
&=&\mu((n,a))I((\frac{n}{(n,a)},a))g(\frac{n}{(n,a)})
\end{eqnarray*}

For any positive integers $b_1,$ $\cdots$, $b_k$,
\begin{eqnarray*}
\phi\chi_{b_1}\circ\cdots\circ\phi\chi_{b_k}(n)&=&\prod_{p^{\alpha}\parallel n}\{\sum_{j=0}^{\alpha}\mu(p^{\alpha-j})\prod_{i=1}^{k}(\phi\chi_{b_i}*u)(p^j)\}\\
&=&\prod_{p^{\alpha}\parallel n}\{\sum_{j=0}^{\alpha}\mu(p^{\alpha-j})\prod_{i=1}^{k}(\phi\chi_{b_i}*u)(p^j)\}\\
&=&\prod_{p^{\alpha}\parallel n}\{\prod_{i=1}^{k}(\phi\chi_{b_i}*u)(p^{\alpha})-\prod_{i=1}^{k}(\phi\chi_{b_i}*u)(p^{\alpha-1})\}\\
&=&\prod_{p^{\alpha}\parallel n}\prod_{ord_p(b_i)<\alpha}p^{ord_p(b_i)}\prod_{ord_p(b_i)\geq\alpha}(p^{\alpha}-p^{\alpha-1}).
\end{eqnarray*}

Combining these two formulas, we have
\begin{eqnarray*}
&&\big(\mu\chi_{p-1}\circ\phi\chi_{\phi(p_1^{\alpha_1})}\circ\cdots\circ\phi\chi_{\phi(p_k^{\alpha_k})}\big)(\delta)\\
&=&\mu((\delta,p-1))I((\frac{\delta}{(\delta,p-1)},p-1))F(\phi(p_1^{\alpha_1}),\cdots,\phi(p_k^{\alpha_k});\frac{\delta}{(\delta,p-1)}).
\end{eqnarray*}
\end{proof}

\subsection*{Acknowledgments}
This work is supported by the National Natural Science Foundation of China (Grant No. 11571303) and the Natural Science Foundation of Zhejiang Province (No. LY18A010016).

\bibliographystyle{amsplain}

\end{document}